\documentclass{amsart}
\usepackage{amsmath}
\usepackage{amssymb}
\usepackage{amsthm}
\usepackage{epsfig}
\usepackage{amsfonts}
\usepackage{amscd}
\usepackage{mathrsfs}
\usepackage{enumerate}
\usepackage{latexsym}

\newtheorem{theorem}{Theorem}[section]
\newtheorem{lemma}[theorem]{Lemma}

\theoremstyle{definition}
\newtheorem{definition}[theorem]{Definition}
\newtheorem{example}[theorem]{Example}

\theoremstyle{remark}
\newtheorem{remark}[theorem]{Remark}

\begin{document}

\title[One--point extensions and local topological properties]{One--point extensions and local topological properties}

\author{M.R. Koushesh}
\address{Department of Mathematical Sciences, Isfahan University of Technology, Isfahan 84156--83111, Iran}
\address{School of Mathematics, Institute for Research in Fundamental Sciences (IPM), P.O. Box: 19395--5746, Tehran, Iran}
\email{koushesh@cc.iut.ac.ir}
\thanks{This research was in part supported by a grant from IPM (No. 90030052).}



\begin{abstract}
A space $Y$ is called an {\em extension} of a space $X$ if $Y$ contains $X$ as a dense subspace. An extension $Y$ of $X$
is called a {\em one--point extension} of $X$ if $Y\backslash X$ is a singleton. P. Alexandroff proved that any locally compact non--compact Hausdorff space $X$ has a one--point compact Hausdorff extension, called the {\em one--point compactification} of $X$. Motivated by this, S. Mr\'{o}wka and J.H. Tsai [On local topological properties. II, \emph{Bull. Acad. Polon. Sci. S\'{e}r. Sci. Math. Astronom. Phys.} {\bf 19} (1971), 1035--1040] posed the following more general question: For what pairs of topological properties ${\mathscr P}$ and ${\mathscr Q}$ does a locally--${\mathscr P}$ space $X$ having ${\mathscr Q}$ possess a one--point extension having both ${\mathscr P}$ and ${\mathscr Q}$? Here, we provide an answer to this old question.
\end{abstract}

\maketitle

\section{Introduction}

Let ${\mathscr P}$ be a topological property. Then
\begin{itemize}
  \item ${\mathscr P}$ is {\em closed  hereditary} if any closed subspace of a space with ${\mathscr P}$, also has ${\mathscr P}$.
  \item ${\mathscr P}$ is {\em preserved under finite closed sums} if any space which is expressible as a finite union of closed subspaces each having ${\mathscr P}$, also has ${\mathscr P}$.
  \item ${\mathscr P}$ {\em satisfies Mr\'{o}wka's condition $(\mbox{\em W})$} if it satisfies the following: If $X$ is a completely regular space in which there exists a point  $p$ with an open  base ${\mathscr B}$ for $X$ at $p$ such that $X\backslash  B$ has ${\mathscr P}$ for any $B\in {\mathscr B}$, then $X$ has ${\mathscr P}$. (See \cite{Mr}.)
\end{itemize}

\begin{remark}
If ${\mathscr P}$ is a topological property which is closed hereditary and productive then Mr\'{o}wka's condition $(\mbox{W})$ is equivalent to the following condition: If a completely regular space $X$ is the union of a compact space and a space with ${\mathscr P}$, then $X$ has ${\mathscr P}$. (See \cite{MRW1}.)
\end{remark}

Let $X$ be a space and let ${\mathscr P}$ be a topological property. The space $X$ is called {\em locally--${\mathscr P}$} if each of its points has a neighborhood in $X$ with ${\mathscr P}$. Note that if $X$ is regular and ${\mathscr P}$ is closed hereditary, then $X$ is locally--${\mathscr P}$ if and only if each $x\in X$ has an open neighborhood $U$ in $X$ such that $\mbox{cl}_XU$ has ${\mathscr P}$.

Let $X$ and $E$ be Hausdorff spaces. The space $X$ is said to be {\em $E$--completely regular} if $X$ is homeomorphic to a subspace of a product $E^\alpha$ for some cardinal $\alpha$. (See \cite{EM} and \cite{Mr1}.) In \cite{MT} (see also \cite{T}) the authors proved that for a topological property ${\mathscr P}$ which is regular--closed hereditary and preserved under finite closed sums, and satisfies Mr\'{o}wka's condition $(\mbox{W})$, every $E$--completely regular (where $E$ is regular and subject to some restrictions) locally--${\mathscr P}$ space has a one--point $E$--completely regular extension having ${\mathscr P}$. (See \cite{Ma} for related results.) The authors then posed the following more general question: {\em For what pairs of topological properties ${\mathscr P}$ and ${\mathscr Q}$ is it true that every locally--${\mathscr P}$ space having ${\mathscr Q}$ has a one--point extension having both ${\mathscr P}$ and ${\mathscr Q}$?} Indeed, the systematic study of this sort of questions dates back to earlier times when P. Alexandroff proved that every locally compact non--compact Hausdorff space has a one--point compact Hausdorff extension (thus answering the question in the case when ${\mathscr P}$ is compactness and ${\mathscr Q}$ is the Hausdorff property). Since then the question has been considered by various authors for specific choices of topological properties ${\mathscr P}$ and ${\mathscr Q}$. In this note we provide an answer to the above old question of S. Mr\'{o}wka and J.H. Tsai. (See also Theorem 4.1 of \cite{MRW} for a related result.) Results of this note modify and simplify those we have proved in the final chapter of \cite{Kou}.

We now review some notation and terminologies. For undefined terms and notation we refer to \cite{E}.

Let $X$ be a space. If $f:X\rightarrow\mathbb{R}$ is continuous, denote $\mbox{Coz}(f)=X\backslash f^{-1}(0)$. Let
\[\mbox{Coz}(X)=\big\{\mbox{Coz}(f):f:X\rightarrow\mathbb{R}\mbox{ is continuous}\big\}.\]

Let $X$ be a completely regular space. The {\em Stone--\v{C}ech compactification} $\beta X$ of $X$ is the compactification of $X$ characterized among all compactifications of $X$ by the following property: Every continuous $f:X\rightarrow [0,1]$ is continuously extendable over $\beta X$; denote by $f_\beta$ this continuous extension of $f$.

\section{One--point ${\mathscr P}$--${\mathscr Q}$--extensions of locally--${\mathscr P}$ non--${\mathscr P}$ ${\mathscr Q}$--spaces}

 The following subspace of $\beta X$, introduced in \cite{Kou} (also in \cite{Ko4}), plays a crucial role.

\begin{definition}\label{RRA}
For a completely regular space $X$ and a topological property ${\mathscr P}$, let
\[\lambda_{\mathscr P} X=\bigcup\big\{\mbox{int}_{\beta X}\mbox{cl}_{\beta X}C:C\in\mbox{Coz}(X)\mbox{ and }\mbox{cl}_XC \mbox{ has }{\mathscr P}\big\}.\]
\end{definition}

\begin{remark}
If ${\mathscr P}$ is pseudocompactness then
\[\lambda_{\mathscr P} X=\mbox{int}_{\beta X}\upsilon X\]
where $\upsilon X$ is the Hewitt realcompactification of $X$. (See \cite{Ko4}, also \cite{Ko7}.)
\end{remark}

If $X$ is a space and $D$ is a dense subspace of $X$, then $\mbox{cl}_XU=\mbox{cl}_X(U\cap D)$ for every open subspace $U$ of $X$; we have the following simple observation.

\begin{lemma}\label{LKG}
Let $X$ be a completely regular space and let $f:X\rightarrow[0,1]$ be continuous. If $0<r<1$ then
\[f_\beta^{-1}\big[[0,r)\big]\subseteq\mbox{\em int}_{\beta X}\mbox{\em cl}_{\beta X}f^{-1}\big[[0,r)\big].\]
\end{lemma}

\begin{proof}
Note that
\[f_\beta^{-1}\big[[0,r)\big]\subseteq\mbox{cl}_{\beta X}f_\beta^{-1}\big[[0,r)\big]=\mbox{cl}_{\beta X}\big(X\cap f_\beta^{-1}\big[[0,r)\big]\big)=\mbox{cl}_{\beta X}f^{-1}\big[[0,r)\big].\]
\end{proof}

\begin{lemma}\label{BBV}
Let $X$ be a completely regular locally--${\mathscr P}$ space, where ${\mathscr P}$ is a closed hereditary topological property. Then $X\subseteq\lambda_{\mathscr P} X$.
\end{lemma}

\begin{proof}
Let $x\in X$ and let $U$ be an open neighborhood of $x$ in $X$ whose closure $\mbox{cl}_XU$ has ${\mathscr P}$. Let $f:X\rightarrow[0,1]$ be continuous with $f(x)=0$ and $f|(X\backslash U)\equiv 1$. Let $C=f^{-1}[[0,1/2)]\in\mbox{Coz}(X)$. Then $C\subseteq U$ and thus $\mbox{cl}_XC$ has ${\mathscr P}$, as it is  closed in $\mbox{cl}_XU$. Therefore $\mbox{int}_{\beta X}\mbox{cl}_{\beta X}C\subseteq\lambda_{\mathscr P} X$. But then $x\in\lambda_{\mathscr P} X$, as $x\in f_\beta^{-1}[[0,1/2)]$ and  $f_\beta^{-1}[[0,1/2)]\subseteq\mbox{int}_{\beta X}\mbox{cl}_{\beta X}C$ by Lemma \ref{LKG}.
\end{proof}

\begin{theorem}\label{HGAB}
Let ${\mathscr P}$ be a closed hereditary topological property preserved under finite closed sums and satisfying Mr\'{o}wka's condition $(\mbox{\em W})$. Let ${\mathscr Q}$ be a closed hereditary topological property satisfying Mr\'{o}wka's condition $(\mbox{\em W})$ and implying complete regularity. If $X$ is a locally--${\mathscr P}$ non--${\mathscr P}$ space having ${\mathscr Q}$ then $X$ has a one--point extension having both ${\mathscr P}$ and ${\mathscr Q}$.
\end{theorem}

\begin{proof}
Let $X$ be a locally--${\mathscr P}$ non--${\mathscr P}$ space having ${\mathscr Q}$. Note that $\lambda_{\mathscr P} X\neq\beta X$; as otherwise, by compactness and the definition of $\lambda_{\mathscr P} X$ we have
\begin{equation}\label{FRTT}
\beta X=\mbox{int}_{\beta X}\mbox{cl}_{\beta X}C_1\cup\cdots\cup\mbox{int}_{\beta X}\mbox{cl}_{\beta X}C_n
\end{equation}
where $C_1,\ldots,C_n\in\mbox{Coz}(X)$ and each $\mbox{cl}_XC_1,\ldots,\mbox{cl}_XC_n$ has ${\mathscr P}$. Taking the intersection of both sides of (\ref{FRTT}) with $X$, we then have
\[X=\mbox{cl}_XC_1\cup\cdots\cup\mbox{cl}_XC_n.\]
This implies that $X$ has ${\mathscr P}$, as it is the finite union of closed subspaces each having ${\mathscr P}$. This is a contradiction. Note that $\lambda_{\mathscr P} X$ is open in $\beta X$ by its definition, and $X\subseteq\lambda_{\mathscr P} X$ by Lemma \ref{BBV}, as $X$ is locally--${\mathscr P}$. Let $T$ be the quotient space of $\beta X$ obtained by contracting the non--empty set $\beta X\backslash\lambda_{\mathscr P} X$ to a point $p$ and denote by $q:\beta X\rightarrow T$ its quotient mapping. Note that $T$ is Hausdorff, as $\beta X\backslash\lambda_{\mathscr P} X$ is closed in the normal space $\beta X$. Since $T$ is also compact, as it is the continuous image of $\beta X$, it is then completely regular. Also, note that $T$ contains $X$ as a dense subspace. Consider the subspace $Y=X\cup\{p\}$ of $T$. Then $Y$ is a completely regular one--point extension of $X$. We need to show that $Y$ has both ${\mathscr P}$ and ${\mathscr Q}$. To show this, since ${\mathscr P}$ and ${\mathscr Q}$ both satisfy Mr\'{o}wka's condition $(\mbox{W})$ it suffices to show that $Y\backslash V$ has ${\mathscr P}$ and ${\mathscr Q}$ for every open neighborhood $V$ of $p$ in $Y$. Let $V$ be an open neighborhood of $p$ in $Y$. Let $V'$ be open in $T$ such that $V=Y\cap V'$. Note that
\[Y\backslash V=X\cap(T\backslash V')=X\cap q^{-1}[T\backslash V'].\]
Since $p\in V'$ we have
\[q^{-1}[T\backslash V']\cap(\beta X\backslash\lambda_{\mathscr P} X)=q^{-1}[T\backslash V']\cap q^{-1}(p)=\emptyset\]
and thus $q^{-1}[T\backslash V']\subseteq\lambda_{\mathscr P} X$. Since $q^{-1}[T\backslash V']$ is compact, as it is closed in $\beta X$, we have
\begin{equation}\label{FT}
q^{-1}[T\backslash V']\subseteq\mbox{int}_{\beta X}\mbox{cl}_{\beta X}D_1\cup\cdots\cup\mbox{int}_{\beta X}\mbox{cl}_{\beta X}D_m
\end{equation}
for some $D_1,\ldots,D_m\in\mbox{Coz}(X)$ such that each $\mbox{cl}_XD_1,\ldots,\mbox{cl}_XD_m$ has ${\mathscr P}$. Taking the intersection of both sides of (\ref{FT}) with $X$, we have
\[Y\backslash V\subseteq\mbox{cl}_XD_1\cup\cdots\cup\mbox{cl}_XD_m=H.\]
But $H$ has ${\mathscr P}$, as it is the finite union of closed subspaces each having ${\mathscr P}$. Therefore $Y\backslash V$ has ${\mathscr P}$, as it is closed in $H$. That $Y\backslash V$ has ${\mathscr Q}$ follows, as $Y\backslash V$ is closed in $X$ and $X$ has ${\mathscr Q}$.
\end{proof}

\begin{example}
The list of topological properties satisfying the assumption of Theorem \ref{HGAB} is quite long and include almost all important covering properties (i.e., topological properties described in terms of the existence of certain kinds of open subcovers or refinements of a given open cover of a certain type), among them are: compactness, countable compactness (more generally, $[\theta,\kappa]$--compactness), the Lindel\"{o}f property (more generally, the $\mu$--Lindel\"{o}f property), paracompactness, metacompactness, countable paracompactness, subparacompactness, submetacompactness (or $\theta$--refinability) and the $\sigma$--para--Lindel\"{o}f property. (See \cite{Bu}, also \cite{Steph}, for definitions. For the proof that these all satisfy Mr\'{o}wka's condition $(\mbox{W})$, see \cite{Kou}. That these topological properties are closed hereditary and preserved under finite closed sums, follow from Theorems 7.1, 7.3 and 7.4 of \cite{Bu}.)
\end{example}

\section{Acknowledgements}

The author wishes to thank the anonymous referee for reading the manuscript and the prompt report given within two weeks.

\bibliographystyle{amsplain}

\end{document}